\documentclass[a4paper,11pt]{amsart}
\usepackage[english]{babel}
\usepackage{amsmath}
\usepackage{tikz}
\usepackage{hyperref}

\newtheorem{thm}{Theorem}
\newtheorem{Lemma}[thm]{Lemma}
\newtheorem{Proposition}[thm]{Proposition}
\newtheorem{Corollary}[thm]{Corollary}

\theoremstyle{remark}
\newtheorem{rem}[thm]{Remark}

\theoremstyle{definition}
\newtheorem{Example}[thm]{Example}

\definecolor{wwwwww}{rgb}{0.4,0.4,0.4}

\hypersetup{pdfpagemode=UseNone}
\hypersetup{pdfstartview=FitH}
		
\begin{document}
\thanks{This study was financed in part by the Coordena\c{c}\~ao de Aperfei\c{c}oamento de Pessoal de N\'ivel Superior - Brasil (CAPES) Finance Code 001- from September, 2016 to March, 2021}
\title[\resizebox{6.1in}{!}{Constant weighted mean curvature hypersurfaces in Shrinking Ricci Solitons}]{Constant weighted mean curvature hypersurfaces in Shrinking Ricci Solitons}

\author[Igor Miranda]{Igor Miranda}
\address{\sc Igor Miranda\\
Instituto de Matem\'atica e Estat\'istica, Universidade Federal Fluminense, Campus Gragoat\'a, Niter\'oi, RJ \\
24210-200,\\ Brazil}
\email{igor\_miranda@id.uff.br}

\author[Matheus Vieira]{Matheus Vieira}
\address{\sc Matheus Vieira\\
	Departamento de Matem\'atica, Universidade Federal do Esp\'irito Santo, Vit\'oria, ES\\
	29075-910\\ Brazil}
\email{matheus.vieira@ufes.br}

\subjclass[2010]{MSC 53C42 \and MSC 53C44}
\keywords{level sets, constant weighted mean curvature, f-minimal, shrinking Ricci soliton, $\lambda$-hypersurfaces}

\begin{abstract}
In this paper, we study constant weighted mean curvature hypersurfaces in shrinking Ricci solitons. First, we show that a constant weighted mean curvature hypersurface with finite weighted volume cannot lie in a region determined by a special level set of the potential function, unless it is the level set. Next, we show that a compact constant weighted mean curvature hypersurface with a certain upper bound or lower bound on the mean curvature is a level set of the potential function. We can apply both results to the cylinder shrinking Ricci soliton ambient space. Finally, we show that a constant weighted mean curvature hypersurface in the Gaussian shrinking Ricci soliton (not necessarily properly immersed) with a certain assumption on the integral of the second fundamental form must be a generalized cylinder.
\end{abstract}

\maketitle 
\section{Introduction}

Let $\left(\bar{M}^{n+1},\bar{g},f\right)$ be a smooth metric measure
space and let $M^{n}$ be a oriented hypersurface in $\bar{M}^{n+1}$. The weighted mean curvature vector $\vec{H}_{f}$ is defined by
\[
\vec{H}_{f}=\vec{H}+\left(\bar{\nabla}f\right)^{\perp}
\]
and the weighted mean curvature $H_f$ is defined by $\vec{H}_{f}= -H_fN$, where $\perp$ is the projection to the normal bundle and $N$ is the unit normal vector. A hypersurface $M^n$ is said to have constant weighted mean curvature (CWMC) if $H_f$ is a constant (see Section \ref{s2}). These hypersurfaces are also known as $\lambda$-hypersurfaces where $\lambda = H_f$. CWMC hypersurfaces can be seen as critical points of the weighted area functional with respect to weighted volume-preserving variations (see \cite{cheng2018complete}, \cite{mcgonagle2015hyperplane}). When $H_f = 0$ they are called $f$-minimal hypersurfaces. For certain choices of $f$, $f$-minimal hypersurfaces are very important singularities of the mean curvature flow in $\mathbb{R}^{n+1}$, namely self-shrinkers (see p.758 and p.768 in \cite{colding2012generic}), translating solitons (see p.153 and p.154 in \cite{hoffman2017notes}) and self-expanders (see p.9016 and p.9017 in \cite{bernstein2021smooth}). It turns out that $\left(\mathbb{R}^{n+1},\bar{g}_{\text{can}},f\right)$ is a Ricci soliton for each of these choices of $f$ (see Section \ref{s2}). Thus the study of self-shrinkers, translating solitons and self-expanders in $\mathbb{R}^{n+1}$ is equivalent to the study of $f$-minimal hypersurfaces in the shrinking, steady and expanding Ricci soliton $\left(\mathbb{R}^{n+1},\bar{g}_{\text{can}},f\right)$ respectively.

There has been a great interest in CWMC hypersurfaces in the Gaussian shrinking Ricci soliton $\left(\mathbb{R}^{n+1},\bar{g}_{\text{can}},f\right)$ with $f(x)=\frac{|x|^2}{4}$. In \cite{mcgonagle2015hyperplane} Mcgonagle-Ross classified stable CWMC hypersurfaces properly immersed in the Gaussian shrinking Ricci soliton. They also proved that there are no CWMC hypersurfaces properly immersed in this ambient space with index one. In \cite{cheng2016rigidity} Q.M.Cheng-Ogata-Wei classified complete CWMC hypersurfaces in the Gaussian shrinking Ricci soliton with a certain condition on the norm of the second fundamental form and the mean curvature. In \cite{cheng2018complete} Q.M.Cheng-Wei classified complete CWMC hypersurfaces embedded in the Gaussian shrinking Ricci soliton with polynomial volume growth and $H-H_f \geq 0$. In \cite{guang2018gap} Guang classified compact CWMC surfaces embedded in the Gaussian shrinking Ricci soliton with $H_f \geq 0$ and constant norm of the second fundamental form. He also classified complete CWMC hypersurfaces embedded in the Gaussian shrinking Ricci soliton with a certain condition on the norm of the second fundamental form. In \cite{cheng2021complete} Q.M.Cheng-Wei generalized this result to complete CWMC surfaces and removed the condition on $H_f$.

There has also been a great interest in f-minimal and CWMC hypersurfaces in the cylinder shrinking Ricci soliton $\left(\mathbb{R}^{n+1-k}\times S_{\sqrt{2\left(k-1\right)}}^{k},\bar{g},f\right)$, where $\bar{g}$ is the product metric and $f\left(x,y\right)=\frac{\left|x\right|^{2}}{4}$ (here $x$ is the position vector in $\mathbb{R}^{n-k+1}$ and $y$ is the position vector in $\mathbb{R}^{k+1}$). In \cite{cheng2015simons} X.Cheng-Mejia-Zhou classified compact f-minimal hypersurfaces in a cylinder shrinking Ricci soliton with a certain condition on the norm of the second fundamental form. They also classified compact f-minimal hypersurfaces in this ambient space with index one. In \cite{cheng2015stability} X.Cheng-Zhou generalized these results to complete hypersurfaces. In \cite{barbosa2020lambda} Barbosa-Santana-Upadhyay obtained theorems for CWMC hypersurfaces similar to the results of Mcgonagle-Ross \cite{mcgonagle2015hyperplane} but in the cylinder shrinking Ricci soliton ambient space. They also obtained theorems for CWMC hypersurfaces in cylinder shrinking Ricci soliton similar to the results of X.Cheng-Mejia-Zhou \cite{cheng2015simons} and X.Cheng-Zhou \cite{cheng2015stability} about $L_f$-stable (which are those that the second variation of weighted area is non-negative) $f$-minimal hypersurfaces.

In this paper we study geometric properties and classification results for CWMC hypersurfaces in shrinking Ricci solitons.

Let $\left(\bar{M}^{n+1},\bar{g},f\right)$ be a gradient Ricci soliton
with 
$$
\bar{R}ic+\bar{\nabla}\bar{\nabla}f=\lambda\bar{g},
$$
where $\lambda$ is a constant. We know that
\[
\left|\bar{\nabla}f\right|^{2}+\bar{R}-2\lambda f=C,
\]
where $\bar{R}$ is the scalar curvature of $\bar{M}^{n+1}$ and $C$
is a constant (see p.85 in \cite{hamilton1993formations} and Lemma 1.1 in \cite{cao2009recent}). Let us define
\[
D^\pm (c_1,c_2)=\left\{ x\in\bar{M}^{n+1};f\left(x\right)<\Gamma^\pm(x)\right\} ,
\]
where
\[
\Gamma^\pm(x)=\frac{1}{2\lambda}\left\lbrace \frac{1}{4}\left(\pm c_1+\sqrt{c_1^{2}+4c_2}\right)^{2}+\bar{R}(x)-C\right\rbrace
\]
and $c_1$ and $c_2$ are constants such that $c_1^2+4c_2\geq 0$.

Note that $D^+(c_1,c_2)$ and $D^-(c_1,c_2)$ are open subsets of $\bar{M}^{n+1}$. In this paper we will always use the above notations and conventions.

In Theorem 3 in \cite{vieira2018geometric}, Vieira-Zhou proved that if $M^n$ is a complete $f$-minimal hypersurface properly immersed in the cylinder shrinking Ricci soliton then: (a) $M^{n}$ cannot lie inside the closed product $\bar{B}^{n+1-k}_{\sqrt{2\left(n-k\right)}}\left(0\right)\times S_{\sqrt{2(k-1)}}^{k}(0)$, unless $M^{n}=S^{n-k}_{\sqrt{2\left(n-k\right)}}\left(0\right)\times S_{\sqrt{2(k-1)}}^{k}(0)$ (see also the earlier work of Cavalcante-Espinar \cite{10.1112/blms/bdv099} for CWMC hypersurfaces in the Gaussian shrinking Ricci soliton $\mathbb{R}^{n+1}$); (b) $M^{n}$ cannot lie outside the product $B^{n+1-k}_{\sqrt{2\left(n+1-k\right)}}\left(0\right)\times S_{\sqrt{2(k-1)}}^{k}(0)$. We generalize this result in two different ways: we extend the result to CWMC hypersurfaces and we consider an arbitrary shrinking Ricci soliton ambient space.

\begin{thm}\label{t3}
	Let $\bar{M}^{n+1}$ be a shrinking Ricci soliton and let $M^{n}$
	be a complete CWMC hypersurface immersed in $\bar{M}^{n+1}$ with finite weighted
	volume.
	
	(a) Suppose that $tr_{M^{n}}\bar{\nabla}\bar{\nabla}f\geq a$ for some
	$a>0$. If $M^{n}$ lies in $\overline{D^-(|H_f|,a)}$ then $M^{n}\subseteq \partial D^-(|H_f|,a)$.
	
	(b) Suppose that $tr_{M^{n}}\bar{\nabla}\bar{\nabla}f\leq b$ for some
	$b> 0$. If $M^{n}$ lies outside $D^+(|H_f|,b)$ then $M^{n}\subseteq \partial D^+(|H_f|,b)$.
\end{thm}

We remark that $\bar{\nabla}\bar{\nabla}f$ is a $(0,2)$ tensor in $\bar{M}^{n+1}$, so we can consider its restriction to $M^n$ (in this case the trace $tr_{M^{n}}\bar{\nabla}\bar{\nabla}f$ is a function in $M^n$). Note that we can state Theorem \ref{t3} in a different way using the fact that
	\[
	tr_{M^{n}}\bar{\nabla}\bar{\nabla}f=n\lambda-\bar{R}+\bar{R}ic\left(N,N\right),
	\]
where $N$ is the unit normal vector of $M^n$.

Using the above result together with a classification result for the level sets of the potential function (see Theorem \ref{lema1}) we obtain a new result for the cylinder shrinking Ricci soliton ambient space.

\begin{Corollary}\label{c1}
	Let $M^{n}$ be a complete CWMC hypersurface properly immersed in
	the cylinder shrinking Ricci soliton $\mathbb{R}^{n+1-k}\times S_{\sqrt{2(k-1)}}^{k}(0)$, where $n\geq 3$ and $2\leq k \leq n-1$.
	
	(a) If $M^{n}$ lies inside the closed product $\bar{B}^{n+1-k}_r\left(0\right)\times S_{\sqrt{2(k-1)}}^{k}(0)$ with $r=-|H_{f}|+\sqrt{H_{f}^{2}+2\left(n-k\right)}$, then $H_f\geq 0$ and $M^{n}=S^{n-k}_r\left(0\right)\times S_{\sqrt{2(k-1)}}^{k}(0)$.
	
	(b) $M^{n}$ cannot lie outside the product $B^{n+1-k}_r\left(0\right)\times S_{\sqrt{2(k-1)}}^{k}(0)$ with $r=|H_{f}|+\sqrt{H_{f}^{2}+2\left(n+1-k\right)}$.
\end{Corollary}

In Theorem 3.5 in \cite{guang2018gap}, Guang proved that if $M^n$ is a compact CWMC hypersurface in the Gaussian shrinking Ricci soliton $\mathbb{R}^{n+1}$ satisfying $H_f\geq 0$ and 
\[
|A|^2\leq \frac{1}{2}+\frac{H_f\left(H_f+\sqrt{H_f^2+2n}\right)}{2n},
\]
then $M^n=S^n_r(0)$, where $r=-H_f+\sqrt{H_f^2+2n}$. We generalize this result in two different ways: we assume an upper bound on the mean curvature (this assumption is weaker since $H^2 \leq n |A|^2$) and we consider an arbitrary shrinking Ricci soliton ambient space. Here $A$ is the second fundamental form and $H$ is the mean curvature.
\begin{thm}\label{t7}
	Let $\bar{M}^{n+1}$ be a shrinking Ricci soliton and let $M^n$ be a compact CWMC hypersurface immersed in $\bar{M}^{n+1}$. Suppose that $tr_{M^{n}}\bar{\nabla}\bar{\nabla}f\geq a$ for some $a>0$. Assume that $f(p)=\sup_{M^n}f$ is a regular value of $f$. If
	\[
	H\leq \frac{2H_f-|H_f|+\sqrt{H_f^2+4a}}{2},
	\]
	then $M^n\subseteq\partial D^-(|H_f|,a)$.
\end{thm}
\begin{rem}
    The theorem above generalizes Theorem 3.5 in \cite{guang2018gap}. Indeed, when the ambient space is the Gaussian shrinking Ricci soliton (see Example \ref{ex1}) we have $tr_{M^{n}}\bar{\nabla}\bar{\nabla}f= \frac{n}{2}$ (so $a=\frac{n}{2}$). Assuming that $H_f\geq 0$ and
\[
|A|^2\leq \frac{1}{2}+\frac{H_f\left(H_f+\sqrt{H_f^2+2n}\right)}{2n}= \frac{\left(H_f+\sqrt{H_f^2+2n}\right)^2}{4n}
\]
and using the fact that $\frac{1}{n}H^2\leq |A|^2$ we have
\[
	H\leq \frac{2H_f-|H_f|+\sqrt{H_f^2+2n}}{2}.
\]
Now using Theorem \ref{t7} we conclude that  $M^n = S^n_r(0)$.
\end{rem}
In particular, we obtain a new result for the cylinder shrinking Ricci soliton ambient space.
\begin{Corollary}\label{t10}
	Let $M^n$ be a compact CWMC hypersurface immersed in the cylinder shrinking Ricci soliton $\mathbb{R}^{n-k+1}\times S^k_{\sqrt{2(k-1)}}(0)$, where $n\geq 3$ and $2\leq k \leq n-1$. If $H_f\geq 0$ and
	\[
	H\leq \frac{H_f+\sqrt{H_f^2+2(n-k)}}{2},
	\]
	then $M^n=S^{n-k}_{r}(0)\times S^k_{\sqrt{2(k-1)}}(0)$, where $r= -H_f+\sqrt{H_f^2+2(n-k)}$.	
\end{Corollary}

Note that the assumption on the upper bound of $H$ is sharp because equality holds for $M^n=S^{n-k}_{r}(0)\times S^k_{\sqrt{2(k-1)}}(0)$, where $r= -H_f+\sqrt{H_f^2+2(n-k)}$.

This is a new result even for $f$-minimal hypersurfaces.
\begin{Corollary}
	Let $M^n$ be a compact $f$-minimal hypersurface immersed in the cylinder shrinking Ricci soliton $\mathbb{R}^{n-k+1}\times S^k_{\sqrt{2(k-1)}}(0)$, where $n\geq 3$ and $2\leq k \leq n-1$. If
	\[
	H\leq \frac{\sqrt{2(n-k)}}{2},
	\]
	then $M^n=S^{n-k}_{r}(0)\times S^k_{\sqrt{2(k-1)}}(0)$, where $r=\sqrt{2(n-k)}$.	
\end{Corollary}
In Theorem 1.2 in \cite{cheng2016rigidity}, Q.M.Cheng-Ogata-Wei proved that if $M^n$ is a complete CWMC hypersurface in the Gaussian shrinking Ricci soliton $\mathbb{R}^{n+1}$ with polynomial volume growth and satisfying
\[
\left(H-\frac{H_f}{2}\right)^2\geq \frac{H_f^2}{4}+\frac{n}{2},
\]
then $M^n=S^n_r(0)$, where $r=-H_f+\sqrt{H_f^2+2n}$. We generalize this result to any ambient space which is a smooth metric measure space.
\begin{thm}\label{t9}
Let $(\bar{M}^{n+1}, \bar{g}, f)$ be a smooth measure metric space and let $M^n$ be a complete hypersurface immersed in $\bar{M}^{n+1}$ with finite weighted volume. Suppose that $\text{tr}_{M}\bar{\nabla}\bar{\nabla}f\leq b$ for some $b>0$. If
\[
H\geq \frac{H_f+\sqrt{H_f^2+4b}}{2},
\]
then $M^n\subseteq f^{-1}(\gamma)$ for some $\gamma\in \mathbb{R}$.
\end{thm}

In particular, we recover the result of Q.M.Cheng-Ogata-Wei (see Corollary \ref{corolult}) and we obtain a new result for the cylinder shrinking Ricci soliton ambient space.

\begin{Corollary}\label{corin3}
	Let $M^n$ be a complete CWMC hypersurface properly immersed in the cylinder shrinking Ricci soliton $\mathbb{R}^{n+1-k}\times S^k_{\sqrt{2(k-1)}}(0)$, where $n\geq 3$ and $2\leq k \leq n-1$. Then
	\[
	\inf H< \frac{H_f+\sqrt{H_f^2+2(n+1-k)}}{2}.
	\]	
\end{Corollary} 
This is a new result even for $f$-minimal hypersurfaces.
\begin{Corollary}
	Let $M^n$ be a complete $f$-minimal hypersurface properly immersed in the cylinder shrinking Ricci soliton $\mathbb{R}^{n+1-k}\times S^k_{\sqrt{2(k-1)}}(0)$, where $n\geq 3$ and $2\leq k \leq n-1$. Then
	\[
	\inf H< \frac{\sqrt{2(n+1-k)}}{2}.
	\]	
\end{Corollary} 

In the second part of this paper, we obtain some classification theorems for the Gaussian shrinking Ricci soliton ambient space. Le-Sesum \cite{le2011blow} proved that if $M^n$ is a complete self-shrinker in the Gaussian shrinking Ricci soliton $\mathbb{R}^{n+1}$ with polynomial volume growth and satisfying $|A|^2<1/2$, then $M^n$ is a hyperplane passing through the origin. Later, Cao-Li \cite{cao2013gap} extended this result to arbitrary codimension by showing that if $|A|^2\leq 1/2$, then $M^n$ is a generalized cylinder. Later, Guang \cite{guang2018gap} extended this result to CWMC hypersurfaces by assuming polynomial volume growth and a condition on the norm of the second fundamental form. Many of the classification results for CWMC hypersurfaces assume that the hypersurface has polynomial volume growth in order to use integration techniques. Recently, there has been some papers trying to avoid this assumption by using the Omori-Yau maximum principles (\cite{cheng2016rigidity}, \cite{cheng2015complete}, \cite{pigola2014complete}, etc). We can replace the assumption of polynomial volume growth in Guang's result by an assumption on the integral of the second fundamental form.
\begin{thm}\label{t4}
	Let $M^n$ be a complete embedded CWMC hypersurface in the Gaussian shrinking Ricci soliton $\mathbb{R}^{n+1}$. If $A\in L_f^q(M)$ for some $q\geq1$ and
	\[
	|A|\leq\frac{\sqrt{H_f^2+2}-|H_f|}{2},
	\]
	then either $M^n$ is a hyperplane or it is a generalized cylinder $S^k_r(0)\times \mathbb{R}^{n-k}$, where $1\leq k\leq n$.
\end{thm}

We remark that for $M^n=S^n_r(0)$ the assumption on the upper bound is not satisfied unless $r=\sqrt{2n}$ (that is $H_f=0$) and for $S^k_r(0)\times \mathbb{R}^{n-k}$ the assumption on the upper bound is not satisfied unless $r=\sqrt{2k}$ (that is $H_f=0$). Therefore the result is only sharp in these cases. It would be interesting to find a sharp result for $H_f \neq 0$.

This is a new result even for self-shrinkers.

\begin{Corollary}
	Let $M^n$ be a complete embedded self-shrinker in the Gaussian shrinking Ricci soliton $\mathbb{R}^{n+1}$. If $ A\in L^q_f(M)$ for some $q\geq1$ and $|A|^2\leq1/2$, then either $M^n$ is a hyperplane passing through the origin or it is a generalized cylinder $S^k_r(0)\times \mathbb{R}^{n-k}$, with $1\leq k\leq n$ and $r=\sqrt{2k}$.
\end{Corollary} 
We remark that Ancari-Miranda \cite{ancari2021rigidity} proved recently it is possible to obtain a similar result if $|A|^2\leq 1/2$ and $H\in L^q_f(M)$ for some even number $q\geq2$ which is a weaker assumption when $q$ is even.\\

For self-shrinker submanifolds $M^n$ in $\mathbb{R}^{n+p}$, Q.M.Cheng-Peng \cite{cheng2015complete} proved that if $\sup |A|^2<1/2$, then $M^n$ is a hyperplane in $\mathbb{R}^{n+1}$. For codimension 1 we generalize Q.M.Cheng-Peng's result to CWMC hypersurfaces.

\begin{Corollary}\label{t5}
	Let $M^n$ be a complete embedded CWMC hypersurface in $\mathbb{R}^{n+1}$. If
	\[
	\sup |A|<\frac{\sqrt{H_f^2+2}-|H_f|}{2},
	\]
	then $M^n$ is a hyperplane.
\end{Corollary}

This paper is organized as follows. In section \ref{s2}, we describe our notations and conventions and prove some basic results. In section \ref{s3}, we prove Theorem \ref{t3}, Theorem \ref{t7}, Theorem \ref{t9} and related results. In Section \ref{s4}, we prove Theorem \ref{t4} and related results.

\section{Preliminaries and basics results}\label{s2}

In this section we describe our notations and conventions and prove some basic results.

\emph{Smooth metric measure spaces.} Let $\left(M,g\right)$ be a Riemannian manifold and let $f$ be a
smooth function on $M$. The triple $\left(M,g,f\right)$ is called
a smooth metric measure space. The measure $e^{-f}dvol$ is called
weighted volume. If $u$ and $v$ are functions on $M^n$ the $L_{f}^{2}$
inner product $u$ and $v$ is defined by
\[
\left\langle u,v\right\rangle _{L_{f}^{2}\left(M\right)}=\int_{M}uve^{-f}
\]
and the $L_{f}^{p}$ norm of $u$ is defined by
\[
\left|u\right|_{L_{f}^{p}\left(M\right)}=\left(\int_{M}\left|u\right|^{p}e^{-f}\right)^{\frac{1}{p}},
\]
where $p\geq1$. The operator
\[
\Delta_{f}=\Delta-\left\langle \nabla f,\nabla\cdot\right\rangle 
\]
is called weighted Laplacian. It is well known that the weighted Laplacian
is a densely defined self-adjoint operator in $L_{f}^{2}$, that is,
if $u$ and $v$ are smooth functions on $M$ with compact support we have
\[
\int_{M}\left(\Delta_{f}u\right)ve^{-f}=-\int_{M}\left\langle \nabla u,\nabla v\right\rangle e^{-f}.
\]
The Bakry-Émery-Ricci curvature is defined by 
\[
Ric_{f}=Ric+\nabla\nabla f.
\]
The triple $\left(M,g,f\right)$ is called a gradient Ricci soliton
if 
\[
Ric_{f}=\lambda g,
\]
where $\lambda$ is a constant. If $\lambda$ is positive, zero or negative
it is called shrinking, steady or expanding, respectively.

\emph{Hypersurfaces in smooth metric measure spaces.} Let $\left(\bar{M}^{n+1},\bar{g},f\right)$ be a smooth metric measure space
and let $M^n$ be a oriented hypersurface of $\bar{M}^{n+1}$. The second
fundamental form is defined by 
\[
A\left(u,v\right)=\langle\bar{\nabla}_{u}v,N\rangle,
\]
where $u$ and $v$ are vector fields on $M^n$, $\bar{\nabla}$
is the Riemannian connection of $\bar{M}^{n+1}$ and $N$ is the unit normal vector.  The mean curvature vector is defined
by 
\[
\vec{H}=(tr_{M^n} A)N,
\]
The weighted mean curvature vector is defined by 
\[
\vec{H}_{f}=\vec{H}+\left(\bar{\nabla}f\right)^{\perp}.
\]
We remark that
\[
\vec{H}_{f}=\sum_{i=1}^n \langle\bar{\nabla}_{e_i}e_i,N\rangle N + \langle \bar{\nabla}f,N \rangle N.
\]
We see that $\vec{H}_{f}$ does not depend on the choice of the normal vector. The weighted mean curvature is defined by
\[
\vec{H}_{f}= -H_fN.
\]
A hypersurface is said to have constant weighted mean curvature (CWMC) if the weighted mean curvature is constant. Note that $\left(M^n,g,f\right)$ is also a smooth metric measure space with weighted measure $e^{-f}dvol_{M}$
and weighted Laplacian
\[
\Delta_{f}=\Delta-\left\langle \nabla f,\nabla\cdot\right\rangle ,
\]
where $g=\bar{g}|M$, $\nabla$ is the Riemannian connection of $M^n$
and $\Delta$ is the Laplacian of $M^n$.

\begin{Example}\label{ex1}
	Let $\bar{M}^{n+1}=\mathbb{R}^{n+1}$, $\bar{g}=\bar{g}_{can}$ and $f\left(x\right)=\frac{\left|x\right|^{2}}{4}$.
	The triple $\left(\mathbb{R}^{n+1},\bar{g},f\right)$ is a gradient shrinking
	Ricci soliton with Bakry-Émery-Ricci curvature 
	\[
	\bar{R}ic_{f}=\frac{1}{2}\bar{g}.
	\]
	A hypersurface $M^n$ of $\mathbb{R}^{n+1}$ is a CWMC hypersurface if and
	only if 
	\[
	H=\frac{\langle x,N\rangle}{2} +H_f.
	\]
	Note that $H_f=0$ if and only if $M^n$ is a self-shrinker. The weighted volume of $M^n$ is $e^{-\frac{|x|^2}{4}}dvol_M$ and the weighted Laplacian of $M^n$ is given by 
	\[
	\Delta_{f}=\Delta-\frac{1}{2}\left\langle x,\nabla\cdot\right\rangle .
	\]
	Note that $\Delta_{f}$ is the operator $\mathcal{L}$ introduced by Colding and Minicozzi \cite{colding2012generic}. We say that $\left(\mathbb{R}^{n+1},\bar{g},f\right)$ is the Gaussian shrinking Ricci soliton.
\end{Example}
\begin{Example}
	Let $\bar{M}^{n+1}=\mathbb{R}^{n+1}$, $\bar{g}=\bar{g}_{can}$ and $f\left(x\right)=-\frac{\left|x\right|^{2}}{4}$.
	The triple $\left(\mathbb{R}^{n+1},\bar{g},f\right)$ is a gradient expanding
	Ricci soliton with Bakry-Émery-Ricci curvature 
	\[
	\bar{R}ic_{f}=-\frac{1}{2}\bar{g}.
	\]
	A hypersurface $M^n$ of $\mathbb{R}^{n+1}$ is a CWMC hypersurface if and
	only if 
	\[
	H=-\frac{\langle x,N\rangle}{2} +H_f.
	\]
	Note that $H_f=0$ if and only if $M^n$ is a self-expander. The weighted volume of $M^n$ is $e^{\frac{|x|^2}{4}}dvol_M$ and the weighted Laplacian of $M^n$ is given by 
	\[
	\Delta_{f}=\Delta+\frac{1}{2}\left\langle x,\nabla\cdot\right\rangle .
	\]
\end{Example}
\begin{Example}
	Let $\bar{M}^{n+1}=\mathbb{R}^{n+1}$, $\bar{g}=\bar{g}_{can}$ and $f\left(x\right)=\langle a, x\rangle$, where $a\in\mathbb{R}^{n+1}$.
	The triple $\left(\mathbb{R}^{n+1},\bar{g},f\right)$ is a gradient steady
	Ricci soliton with Bakry-Émery-Ricci curvature 
	\[
	\bar{R}ic_{f}=0.
	\]
	A hypersurface $M^n$ of $\mathbb{R}^{n+1}$ is a CWMC hypersurface if and
	only if 
	\[
	H=\langle a,N\rangle +H_f.
	\]
	Note that $H_f=0$ if and only if $M^n$ is a translating soliton. The weighted volume of $M^n$ is $e^{-\langle a, x\rangle}dvol_M$ and the weighted Laplacian of $M^n$ is given by 
	\[
	\Delta_{f}=\Delta-\left\langle a,\nabla\cdot\right\rangle .
	\]
	
	In \cite{lopez2018invariant} Lopez classified CWMC surfaces in the steady Ricci soliton $\mathbb{R}^3$ that are invariant by rotations and translations.
\end{Example}
\begin{Example}\label{ex2}
	Let $\bar{M}^{n+1}=\mathbb{R}^{n+1-k}\times S_{\sqrt{2\left(k-1\right)}}^{k}$
	with product metric $\bar{g}$ and potential function $f\left(x,y\right)=\frac{\left|x\right|^{2}}{4}$.
	Here $x$ is the position vector in $\mathbb{R}^{n-k+1}$ and $y$ is
	the position vector in $\mathbb{R}^{k+1}$. The triple $\left(\mathbb{R}^{n+1-k}\times S_{\sqrt{2\left(k-1\right)}}^{k},\bar{g},f\right)$
	is a gradient shrinking Ricci soliton with Bakry-Émery-Ricci curvature
	\begin{align*}
	\bar{R}ic_{f} & =\bar{\nabla}\bar{\nabla}f+\bar{R}ic\\
	& =\frac{1}{2}g_{\mathbb{R}^{n+1-k}}+\frac{1}{2}g_{S_{\sqrt{2\left(k-1\right)}}^{k}}\\
	& =\frac{1}{2}\bar{g}.
	\end{align*}
	A hypersurface $M^n$ of $\mathbb{R}^{n+1-k}\times S_{\sqrt{2\left(k-1\right)}}^{k}$
	is a CWMC hypersurface if and only if 
	\[
	H=\frac{\langle x,N\rangle}{2} +H_f.
	\]
	The weighted volume of $M^n$ is $e^{-\frac{|x|^2}{4}}dvol_M$ and the weighted Laplacian of $M^n$ is given by 
	\[
	\Delta_{f}=\Delta-\frac{1}{2}\left\langle x,\nabla\cdot\right\rangle .
	\]
	We say that $\left(\mathbb{R}^{n+1-k}\times S_{\sqrt{2\left(k-1\right)}}^{k},\bar{g},f\right)$ is the cylinder shrinking Ricci soliton.
\end{Example}

In the rest of the section we prove some results for the level set of the potential function of a shrinking Ricci soliton.

\begin{thm}\label{lema1}
	Let $\bar{M}^{n+1}$ be a complete gradient shrinking Ricci soliton
	with constant scalar curvature. Suppose that
	\[
	M^{n}=\left\{ x\in\bar{M}^{n+1}:f\left(x\right)=\gamma\right\} ,
	\]
	where $\gamma$ is a regular value of $f$. Then $M^n$ is a CWMC hypersurface with
	\[
	H_{f}=\frac{n\lambda-2\lambda\gamma-C}{\sqrt{2\lambda\gamma-\bar{R}+C}},
	\]
	and
	\[
	\gamma=\frac{1}{2\lambda}\left\lbrace \frac{1}{4}\left(-H_{f}+\sqrt{H_{f}^{2}+4\left(n\lambda-\bar{R}\right)}\right)^{2}+\bar{R}-C\right\rbrace .
	\]
	Moreover $\bar{R}\in\lbrace 0, 2\lambda, ..., n\lambda\rbrace$.
\end{thm}
\begin{proof}
	Using the fact that $N=\frac{\bar{\nabla}f}{\left|\bar{\nabla}f\right|}$ and $\bar{\nabla}\bar{\nabla}f=-\bar{R}ic+\lambda\bar{g}$
	we have
	\begin{align*}
	A\left(v,w\right) & =-\left\langle \bar{\nabla}_{v}N,w\right\rangle \\
	& =-\left\langle \bar{\nabla}_{v}\frac{\bar{\nabla}f}{\left|\bar{\nabla}f\right|},w\right\rangle \\
	& =-\frac{1}{\left|\bar{\nabla}f\right|}\bar{\nabla}\bar{\nabla}f\left(v,w\right)\\
	& =\frac{1}{\left|\bar{\nabla}f\right|}\left(\bar{R}ic\left(v,w\right)-\lambda\left\langle v,w\right\rangle \right).
	\end{align*}
	Then
	\begin{align*}
	H_{f} & =H-\left\langle \bar{\nabla}f,N\right\rangle \\
	& =\frac{1}{\left|\bar{\nabla}f\right|}\left(n\lambda-tr_{M^{n}}\bar{R}ic\right)-\left\langle \bar{\nabla}f,\frac{\bar{\nabla}f}{\left|\bar{\nabla}f\right|}\right\rangle \\
	& =\frac{1}{\left|\bar{\nabla}f\right|}\left(n\lambda-\bar{R}+\bar{R}ic\left(N,N\right)-\left|\bar{\nabla}f\right|^{2}\right).
	\end{align*}
	It is well known that if $\bar{M}^{n+1}$ is a gradient Ricci soliton
	then $\bar{R}ic\left(\bar{\nabla}f\right)=\frac{1}{2}\bar{\nabla}\bar{R}$ (see Equation 1.7 in \cite{cao2009recent}).
	Using this and the fact that $\bar{R}$ is constant and $N=\frac{\bar{\nabla}f}{\left|\bar{\nabla}f\right|}$
	we have $\bar{R}ic\left(N,N\right)=0$. We find that
	\[
	H_{f}=\frac{1}{\left|\bar{\nabla}f\right|}\left(n\lambda-\bar{R}-\left|\bar{\nabla}f\right|^{2}\right).
	\]
	Using the fact that $\bar{R}+\left|\bar{\nabla}f\right|^{2}=2\lambda f+C$ we get the first part of the result.
	
	Now we prove the second part of the result. Multiplying by $\left|\bar{\nabla}f\right|$
	we obtain
	\[
	\left|\bar{\nabla}f\right|^{2}+H_{f}\left|\bar{\nabla}f\right|-\left(n\lambda-\bar{R}\right)=0.
	\]
	Solving the quadratic equation we get
	\[
	\left|\bar{\nabla}f\right|=\frac{1}{2}\left(-H_{f}+\sqrt{H_{f}^{2}+4\left(n\lambda-\bar{R}\right)}\right),
	\]
	or
	\[
	\left|\bar{\nabla}f\right|=\frac{1}{2}\left(-H_{f}-\sqrt{H_{f}^{2}+4\left(n\lambda-\bar{R}\right)}\right).
	\]
	Claim: The second case does not happen. We will prove the claim later.
	Assuming the claim we have
	\[
	\left|\bar{\nabla}f\right|^{2}=\frac{1}{4}\left(-H_{f}+\sqrt{H_{f}^{2}+4\left(n\lambda-\bar{R}\right)}\right)^{2}.
	\]
	Using the fact that $\left|\bar{\nabla}f\right|^{2}=2\lambda f-\bar{R}+C$ we get
	\[
	f=\frac{1}{2\lambda}\left\lbrace \frac{1}{4}\left(-H_{f}+\sqrt{H_{f}^{2}+4\left(n\lambda-\bar{R}\right)}\right)^{2}+\bar{R}-C\right\rbrace .
	\]
	This proves the second part of the result.
	
	Now we prove the claim. We only need to show that
	\[
	\left|\bar{\nabla}f\right|=\frac{1}{2}\left(-H_{f}-\sqrt{H_{f}^{2}+4\left(n\lambda-\bar{R}\right)}\right)\leq0.
	\]
	Since $\bar{R}$ is constant on $\bar{M}^{n+1}$  we know that $\bar{R}\in\left\{ 0,2\lambda,...,n\lambda,\left(n+1\right)\lambda\right\}$ (see Theorem 1 in \cite{fernandez2016gradient}). We will show that $\bar{R}\in\left\{ 0,2\lambda,...,n\lambda\right\} $. Suppose that $\bar{R}=\left(n+1\right)\lambda$. In this case by Proposition 3.3 in \cite{petersen2009rigidity} we see that $\bar{M}^{n+1}$ is Einstein, 	which implies that 
	\[
	\bar{R}ic=\frac{\bar{R}}{n+1}\bar{g}=\lambda\bar{g}.
	\]
	Since $\bar{\nabla}\bar{\nabla}f+\bar{R}ic=\lambda\bar{g}$ we have
	$\bar{\nabla}\bar{\nabla}f=0$. This implies that $\left|\bar{\nabla}f\right|$
	is constant on $\bar{M}^{n+1}$. Since $2\lambda f=-\left|\bar{\nabla}f\right|^{2}-\bar{R}+C$
	we see that $f$ is constant on $\bar{M}^{n+1}$ and $\bar{\nabla}f=0$.
	This contradicts the fact that $M^{n}$ is a level set of the potential
	function. Therefore $\bar{R}\in\left\{ 0,2\lambda,...,n\lambda\right\} $.
	Since $n\lambda-\bar{R}\geq0$ we have 
	\[
	-H_{f}-\sqrt{H_{f}^{2}+4\left(n\lambda-\bar{R}\right)}\leq0.
	\]
	This proves the claim.
\end{proof}

\begin{rem}
Note that when the ambient is a shrinking Ricci soliton with scalar curvature $\bar{R}=n\lambda$, there is no level set which is $f$-minimal. Indeed,
	\[
	H_{f}=\frac{1}{\left|\bar{\nabla}f\right|}\left(n\lambda-\bar{R}-\left|\bar{\nabla}f\right|^{2}\right)
	\]
	implies $\left|\bar{\nabla}f\right|=0$, a contradiction.
\end{rem}

Using the normalization $\lambda=1/2$ and $C=\bar{R}$ (which is used in many papers, see for example pag 2 in \cite{cao2009recent}) we have:
\begin{Corollary}\label{cor2.6}
	Let $\bar{M}^{n+1}$ be a complete gradient shrinking Ricci soliton
	with constant scalar curvature. Suppose that
	\[
	M^{n}=\left\{ x\in\bar{M}^{n+1}:f\left(x\right)=\gamma\right\} ,
	\]
	where $\gamma$ is a regular value of $f$. Then $M^n$ is a CWMC hypersurface with
	\[
	H_{f}=\frac{n-2\gamma-2\bar{R}}{2\sqrt{\gamma}},
	\]
	and
	\[
	\gamma= \frac{1}{4}\left(-H_{f}+\sqrt{H_{f}^{2}+4\left(\frac{n}{2}-\bar{R}\right)}\right)^{2} .
	\]
	Moreover $\bar{R}\in\lbrace 0, 1, ..., \frac{n}{2}\rbrace$.
\end{Corollary}
For 3-dimensional shrinking Ricci solitons we have:
 \begin{Corollary}
 	Let $\bar{M}^{3}$ be a complete gradient shrinking Ricci soliton
 	with constant scalar curvature. Suppose that
 	\[
 	M^{2}=\left\{ x\in\bar{M}^{3}:f\left(x\right)=\gamma\right\} ,
 	\]
 	where $\gamma$ is a regular value of $f$. Then $M^2$ is a CWMC hypersurface and $\bar{R}\in\lbrace 0, 1\rbrace$. Moreover
 	\[
 	H_{f}=\frac{1-\gamma}{\sqrt{\gamma}}\ \ \text{in the case}\ \ \bar{R}=0,
 	\]
 	and
 	\[
 	H_{f}=-\sqrt{\gamma}\ \ \text{in the case}\ \ \bar{R}=1.
 	\]
 \end{Corollary}
For 4-dimensional shrinking Ricci solitons we have:
\begin{Corollary}
	Let $\bar{M}^{4}$ be a complete gradient shrinking Ricci soliton
	with constant scalar curvature. Suppose that
	\[
	M^{3}=\left\{ x\in\bar{M}^{4}:f\left(x\right)=\gamma\right\} ,
	\]
 	where $\gamma$ is a regular value of $f$. Then $M^3$ is a CWMC hypersurface and $\bar{R}\in\lbrace 0, 1, \frac{3}{2}\rbrace$. Moreover
    \[
	H_{f}=\frac{3-2\gamma}{2\sqrt{\gamma}},\ \ \text{in the case}\ \ \bar{R}=0,
	\]
	\[
	H_{f}=\frac{1-2\gamma}{2\sqrt{\gamma}},\ \ \text{in the case}\ \ \bar{R}=1,
	\]
	and
	\[
	H_{f}=-\sqrt{\gamma},\ \ \text{in the case}\ \ \bar{R}=\frac{3}{2}.
	\]
\end{Corollary}

Another consequence of Theorem \ref{lema1} is the following.

\begin{Corollary}\label{corintro}
	Let $r>0$. Then $S^{n-k}_r\left(0\right)\times S_{\sqrt{2\left(k-1\right)}}^{k}(0)$
	is a CWMC hypersurface in the cylinder shrinking Ricci soliton $\mathbb{R}^{n+1-k}\times S_{\sqrt{2\left(k-1\right)}}^{k}(0)$
	and $r=-H_{f}+\sqrt{H_{f}^{2}+2\left(n-k\right)}$.
\end{Corollary}
\begin{proof}
	By Example \ref{ex2} we have $\lambda=\frac{1}{2}$, $\bar{R}=\frac{k}{2}$
	and $C=\bar{R}$ (since $C-\bar{R}=\left|\bar{\nabla}f\right|^{2}-2\lambda f$).
	We have
	\[
	S^{n-k}_r\left(0\right)\times S_{\sqrt{2\left(k-1\right)}}^{k}(0)=\left\{ (x,y)\in\mathbb{R}^{n+1-k}\times S_{\sqrt{2\left(k-1\right)}}^{k}(0);f\left(x,y\right)=\gamma\right\} ,
	\]
	where $\gamma=\frac{r^{2}}{4}$. By Theorem \ref{lema1} we have
	\[
	\gamma=\frac{1}{4}\left(-H_{f}+\sqrt{H_{f}^{2}+2\left(n-k\right)}\right)^{2}.
	\]
	This proves the result.
\end{proof}
Using similar computations as in the proof of Theorem \ref{lema1} we show that the level sets cannot be $f$-minimal when the ambient space is a steady or expanding Ricci soliton with constant scalar curvature. 

\begin{Proposition}\label{t1}
	Let $\bar{M}^{n+1}$ be a complete gradient Ricci soliton  with constant scalar curvature. If there exists a regular level set of the potential function f which is f-minimal, then $\bar{M}^{n+1}$ is a gradient shrinking Ricci soliton.
\end{Proposition}
\begin{proof}
Let $M^n = f^{-1}(\gamma)$ be a regular level set of $f$ which is $f$-minimal. By the proof of Theorem \ref{lema1} we have
\[
	H_{f}=\frac{1}{\left|\bar{\nabla}f\right|}\left(n\lambda-\bar{R}-\left|\bar{\nabla}f\right|^{2}\right).
\]
From the assumption that $M^n$ is f-minimal we have
$$n\lambda-\bar{R}-|\bar{\nabla}f|^2 = 0.$$
First suppose that $\bar{M}^{n+1}$ is a steady soliton ($\lambda = 0$). By Theorem 1.3 in \cite{zhang2009completeness} we have $\bar{R} \geq 0$. From the equation above we conclude that $|\bar{\nabla}f|=0$, which is a contradiction with the fact that $M^n$ is a level of the potential function.\\
Now suppose that $\bar{M}^{n+1}$ is an expanding Ricci soliton ($\lambda < 0$). By Theorem 1 in \cite{fernandez2016gradient} we have $\bar{R}\in\lbrace (n+1)\lambda, n\lambda,..., 2\lambda,0\rbrace$. Using the same argument as in the proof of Lemma \ref{lema1} we have $\bar{R}\geq n\lambda$. From the equation above we conclude that $|\bar{\nabla}f|=0$, which is a contradiction with the fact that $M^n$ is a level of the potential function.
\end{proof}

In Proposition 5.3 in \cite{cao2013gap}, Cao-Li showed that there are no compact self-expanders in the expanding Ricci soliton $(\mathbb{R}^{n+1}, \bar{g}_{can}, f=-|x|^2/4)$. Thus, the above result was expected for expanding Ricci solitons with a proper potential function.

\section{Geometric results in shrinking Ricci solitons}\label{s3}

In this section we prove Theorem \ref{t3}, Theorem \ref{t7}, Theorem \ref{t9} and related results.

Proof of Theorem \ref{t3}.

\begin{proof}
		Fact (a). Let $u$ be a function on $M^{n}$. If $u$ is bounded from
		above and $\Delta_{f}u\geq0$ then $u$ is constant.
		
		Fact (b). Let $u$ be a function on $M^{n}$. If $u$ is bounded from
		below and $\Delta_{f}u\leq0$ then $u$ is constant.
		
		Since $M^{n}$ has finite weighted volume both facts follow from Corollary 1 in \cite{cheng1975differential}. We remark that the extension of these results to smooth metric measure spaces is	straightforward. See also Theorem 25 and Remark 26 in \cite{pigola2011remarks}.
		
		We have
		\[
		\nabla\nabla f\left(u,v\right)=\bar{\nabla}\bar{\nabla}f\left(u,v\right)+\left\langle \bar{\nabla}f,\left(\bar{\nabla}_{u}v\right)^{\perp}\right\rangle .
		\]
		We see that
		\begin{align*}
		\Delta_{f}f & =tr_{M^{n}}\bar{\nabla}\bar{\nabla}f+\left\langle \bar{\nabla}f,\vec{H}\right\rangle -\left\langle \nabla f,\nabla f\right\rangle \\
		& =tr_{M^{n}}\bar{\nabla}\bar{\nabla}f+\left\langle \bar{\nabla}f,\vec{H}+\left(\bar{\nabla}f\right)^{\perp}\right\rangle -\left\langle \bar{\nabla}f,\bar{\nabla}f\right\rangle \\
		& =tr_{M^{n}}\bar{\nabla}\bar{\nabla}f+\left\langle \bar{\nabla}f,\vec{H}_{f}\right\rangle -\left|\bar{\nabla}f\right|^{2}.
		\end{align*}
		
		Proof of Item (a). Since $tr_{M^{n}}\bar{\nabla}\bar{\nabla}f\geq a$
		we have
		\[
		\Delta_{f}f\geq a-\left|H_{f}\right|\left|\bar{\nabla}f\right|-\left|\bar{\nabla}f\right|^{2}.
		\]
		Claim: We have
		\[
		a-\left|H_{f}\right|\left|\bar{\nabla}f\right|-\left|\bar{\nabla}f\right|^{2}\geq0\,\,\,on\,\,\,M^{n},
		\]
		and equality holds if and only if $M^{n}\subseteq\partial D^-(|H_f|,a)$. We
		will prove this claim later. By the claim we have
		\[
		\Delta_{f}f\geq a-\left|H_{f}\right|\left|\bar{\nabla}f\right|-\left|\bar{\nabla}f\right|^{2}\geq0.
		\]
		Since $f\leq\Gamma^-$ on $M^{n}$, by Fact (a) we see that $f$ is
		constant on $M^{n}$. We conclude that
		\[
		a-\left|H_{f}\right|\left|\bar{\nabla}f\right|-\left|\bar{\nabla}f\right|^{2}=0\,\,\,on\,\,\,M^{n}.
		\]
		By the claim we have $M^{n}\subseteq\partial D^-(|H_f|,a)$. Now we prove the
		claim. We have
		\begin{align*}
		& a-\left|H_{f}\right|\left|\bar{\nabla}f\right|-\left|\bar{\nabla}f\right|^{2}\geq0\,\,\,on\,\,\,M^{n}\\
		& \iff\left(\left|\bar{\nabla}f\right|+\frac{\left|H_{f}\right|}{2}\right)^{2}\leq\frac{H_{f}^{2}}{4}+a\,\,\,on\,\,\,M^{n}\\
		& \iff\left|\bar{\nabla}f\right|\leq\frac{1}{2}\left(-\left|H_{f}\right|+\sqrt{H_{f}^{2}+4a}\right)\,\,\,on\,\,\,M^{n}.
		\end{align*}
		Since $\left|\bar{\nabla}f\right|^{2}=2\lambda f-\bar{R}+C$ and
		\[
		2\lambda\Gamma^--\bar{R}+C=\frac{1}{4}\left(-\left|H_{f}\right|+\sqrt{H_{f}^{2}+4a}\right)^{2},
		\]
		we have
		\begin{align*}
		M^{n}\subset\overline{D^-(|H_f|,a)} & \iff f\leq\Gamma^-\,\,\,on\,\,\,M^{n}\\
		& \iff2\lambda f-\bar{R}+C\leq2\lambda\Gamma^--\bar{R}+C\,\,\,on\,\,\,M^{n}\\
		& \iff\left|\bar{\nabla}f\right|^{2}\leq\frac{1}{4}\left(-\left|H_{f}\right|+\sqrt{H_{f}^{2}+4a}\right)^{2}\,\,\,on\,\,\,M^{n}\\
		& \iff\left|\bar{\nabla}f\right|\leq\frac{1}{2}\left(-\left|H_{f}\right|+\sqrt{H_{f}^{2}+4a}\right)\,\,\,on\,\,\,M^{n}.
		\end{align*}
		In the last line we used the fact that $a\geq 0$. We conclude that
		\[
		a-\left|H_{f}\right|\left|\bar{\nabla}f\right|-\left|\bar{\nabla}f\right|^{2}\geq0\,\,\,on\,\,\,M^{n}\iff M^{n}\subset\overline{D^-(|H_f|,a)}.
		\]
		Moreover equality holds in the left hand side if and only if $M^{n}\subseteq\partial D^-(|H_f|,a)$.
		This proves the claim.
		
		Proof of Item (b). Since $tr_{M^{n}}\bar{\nabla}\bar{\nabla}f\leq b$
		we have 
		\[
		\Delta_{f}f\leq b+\left|H_{f}\right|\left|\bar{\nabla}f\right|-\left|\bar{\nabla}f\right|^{2}.
		\]
		Claim. We have
		\[
		b+\left|H_{f}\right|\left|\bar{\nabla}f\right|-\left|\bar{\nabla}f\right|^{2}\leq0\,\,\,on\,\,\,M^{n},
		\]
		and equality holds if and only if $M^{n}\subseteq\partial D^+(|H_f|,b)$. We
		will prove this claim later. By the claim we have
		\[
		\Delta_{f}f\leq b+\left|H_{f}\right|\left|\bar{\nabla}f\right|-\left|\bar{\nabla}f\right|^{2}\leq0.
		\]
		Since $f\geq\Gamma^+$ on $M^{n}$, by Fact (b) we see that $f$ is constant
		on $M^{n}$. We conclude that
		\[
		b+\left|H_{f}\right|\left|\bar{\nabla}f\right|-\left|\bar{\nabla}f\right|^{2}=0\,\,\,on\,\,\,M^{n}.
		\]
		By the claim we have $M^{n}\subseteq\partial D^+(|H_f|,b)$. Now we prove the
		claim. We have
		\begin{align*}
		& b+\left|H_{f}\right|\left|\bar{\nabla}f\right|-\left|\bar{\nabla}f\right|^{2}\leq0\,\,\,on\,\,\,M^{n}\\
		& \iff\left(\left|\bar{\nabla}f\right|-\frac{\left|H_{f}\right|}{2}\right)^{2}\geq\frac{H_{f}^{2}}{4}+b\,\,\,on\,\,\,M^{n}\\
		& \iff\left|\bar{\nabla}f\right|\geq\frac{1}{2}\left(\left|H_{f}\right|+\sqrt{H_{f}^{2}+4b}\right)\,\,\,on\,\,\,M^{n}.
		\end{align*}
		Since $\left|\bar{\nabla}f\right|^{2}=2\lambda f-\bar{R}+C$ and
		\[
		2\lambda\Gamma^+-\bar{R}+C=\frac{1}{4}\left(\left|H_{f}\right|+\sqrt{H_{f}^{2}+4b}\right)^{2},
		\]
		we have
		\begin{align*}
		M^{n}\subset\bar{M}^{n+1}\setminus D^+(|H_f|,b) & \iff f\geq\Gamma^+\,\,\,on\,\,\,M^{n}\\
		& \iff2\lambda f-\bar{R}+C\geq2\lambda\Gamma^+-\bar{R}+C\,\,\,on\,\,\,M^{n}\\
		& \iff\left|\bar{\nabla}f\right|^{2}\geq\frac{1}{4}\left(\left|H_{f}\right|+\sqrt{H_{f}^{2}+4b}\right)^{2}\,\,\,on\,\,\,M^{n}\\
		& \iff\left|\bar{\nabla}f\right|\geq\frac{1}{2}\left(\left|H_{f}\right|+\sqrt{H_{f}^{2}+4b}\right)\,\,\,on\,\,\,M^{n}.
		\end{align*}
		In the last line we used the fact that $b\geq0$. We conclude that
		\[
		b+\left|H_{f}\right|\left|\bar{\nabla}f\right|-\left|\bar{\nabla}f\right|^{2}\leq0\,\,\,on\,\,\,M^{n}\iff M^{n}\subset\bar{M}^{n+1}\setminus D^+(|H_f|,b).
		\]
		Moreover equality holds in the left hand side if and only if $M^{n}\subseteq\partial D^+(|H_f|,b)$.
\end{proof}
The following lemma will be important throughout this work in order to guarantee the equivalence between properness and finite weighted volume.
\begin{Lemma} \label{remark1}
Let ($\bar{M}^{n+1},\bar{g},f)$ be a complete gradient shrinking Ricci soliton with constant scalar curvature and let $M^n$ be a complete CWMC hypersurface in $\bar{M}^{n+1}$. Suppose that $\bar{\nabla}\bar{\nabla}f (N,N) \geq k_1$ or $\text{tr}_{M^n}  \bar{\nabla}\bar{\nabla}f\leq k_2$ where $k_1$ and $k_2$ are constants. Then $M^n$ is properly immersed if and only if $M^n$ has finite weighted volume.
\end{Lemma}

\begin{proof} When $\bar{\nabla}\bar{\nabla}f (N,N) \geq k_1$, the result follows from Theorem 1.3 in \cite{cheng2019volume}.

Now suppose that $\text{tr}_{M^n} \bar{\nabla}\bar{\nabla}f \leq k_2$. Using the fact that $\bar{R}ic+\bar{\nabla}\bar{\nabla}f=\lambda\bar{g}$ we have
\begin{eqnarray*}
		(n+1)\lambda&=&\bar{R}+\bar{\Delta}f\\
		            &=&\bar{R}+\text{tr}_{M^n}\bar{\nabla}\bar{\nabla}f+\bar{\nabla}\bar{\nabla}f(N,N).
\end{eqnarray*}
We see that
$$\bar{\nabla}\bar{\nabla}f(N,N) \geq (n+1)\lambda - \bar{R} - k_2.$$
Using Theorem 1.3 in \cite{cheng2019volume} we get the result.
\end{proof}

Proof of Corollary \ref{c1}.

\begin{proof}
By Example \ref{ex2} we have $\lambda=\frac{1}{2}$ and $C=\bar{R}$
(since $C-\bar{R}=\left|\bar{\nabla}f\right|^{2}-2\lambda f$). We
have $\bar{\nabla}\bar{\nabla}f=\frac{1}{2}g_{R^{n+1-k}}$ and
\begin{align*}
tr_{M^{n}}\bar{\nabla}\bar{\nabla}f & =\bar{\Delta}f-\bar{\nabla}\bar{\nabla}f\left(N,N\right)\\
& =\frac{n+1-k}{2}-\bar{\nabla}\bar{\nabla}f\left(N,N\right).
\end{align*}
In particular
\[
\frac{n-k}{2}\leq tr_{M^{n}}\bar{\nabla}\bar{\nabla}f\leq\frac{n+1-k}{2}.
\]
For $D^-(|H_f|,a)$ with $a=\frac{n-k}{2}$ we have
\[
\Gamma^-=\frac{1}{4}\left(-\left|H_{f}\right|+\sqrt{H_{f}^{2}+2\left(n-k\right)}\right)^{2},
\]
and for $D^+(|H_f|, b)$ with $b=\frac{n-k+1}{2}$ we have
\[
\Gamma^+=\frac{1}{4}\left(\left|H_{f}\right|+\sqrt{H_{f}^{2}+2\left(n+1-k\right)}\right)^{2}.
\]
Since $M^{n}$ is properly immersed we know that $M^{n}$ has finite weighted volume (see Lemma \ref{remark1}).

Proof of Item (a). Using the fact that $f(x,y) = \frac{|x|^2}{4}$ we have $D^-(|H_f|,a)=B^{n+1-k}_r\left(0\right)\times S_{\sqrt{2\left(k-1\right)}}^{k}(0)$
and $\partial D^-(|H_f|,a)=S^{n-k}_r\left(0\right)\times S_{\sqrt{2\left(k-1\right)}}^{k}(0)$ where $r=-|H_{f}|+\sqrt{H_{f}^{2}+2\left(n-k\right)}$. Using the fact that $M^{n}$ lies in $\bar{B}^{n+1-k}_r\left(0\right)\times S_{\sqrt{2(k-1)}}^{k}(0)$ and Item (a) of Theorem \ref{t3} we have $M^n\subseteq S^{n-k}_r\left(0\right)\times S_{\sqrt{2\left(k-1\right)}}^{k}(0)$. Since $M^n$ is complete we have $M^n= S^{n-k}_r\left(0\right)\times S_{\sqrt{2\left(k-1\right)}}^{k}(0)$. On the other hand, by Corollary \ref{corintro} we have $r=-H_{f}+\sqrt{H_{f}^{2}+2\left(n-k\right)}$, so $H_f\geq 0$.

Proof of Item (b). Using the fact that $f(x,y) = \frac{|x|^2}{4}$ we have $D^+(|H_f|,b)=B^{n-k+1}_r\left(0\right)\times S_{\sqrt{2(k-1)}}^{k}(0)$ and $\partial D^+(|H_f|,b)=S^{n-k}_r\left(0\right)\times S_{\sqrt{2(k-1)}}^{k}(0)$ where  $r=|H_{f}|+\sqrt{H_{f}^{2}+2\left(n+1-k\right)}$. Using the fact that $M^{n}$ lies outside $B^{n+1-k}_r\left(0\right)\times S_{\sqrt{2(k-1)}}^{k}(0)$ and Item (b) of Theorem \ref{t3} we have $M^{n}\subseteq S^{n-k}_r\left(0\right)\times S_{\sqrt{2(k-1)}}^{k}(0)$. Since $M^n$ is complete we have $M^{n}=S^{n-k}_r\left(0\right)\times S_{\sqrt{2(k-1)}}^{k}(0)$. On the other hand, by Corollary \ref{corintro} we have $r=-H_{f}+\sqrt{H_{f}^{2}+2\left(n-k\right)}$. Therefore
\begin{eqnarray*}
	-H_f+\sqrt{H_f^2+2(n-k)}&=&|H_f|+\sqrt{H_f^2+2(n+1-k)}\\
							&>& |H_f|+\sqrt{H_f^2+2(n-k)}.
\end{eqnarray*}
Hence
\[
-H_f>|H_f|,
\]
which is a contradiction.
\end{proof}

For completeness we include an application of Theorem \ref{t3} to the Gaussian shrinking soliton ambient space. We omit the proof (it is similar to the proof of Corollary \ref{c1}).

\begin{Corollary}
	Let $M^n$ be a complete CWMC hypersurface properly immersed  in the Gaussian shrinking Ricci soliton $\mathbb{R}^{n+1}$.

(a) If $M^n$ lies inside $\bar{B}^{n+1}_{r}(0)$ with $r=-|H_f|+\sqrt{H_f^2+2n}$, then $H_f\geq 0$ and $M^n=S^n_{r}(0)$.

(b) If $M^n$ lies outside $B^{n+1}_{r}(0)$ with $r=|H_f|+\sqrt{H_f^2+2n}$, then $H_f\leq 0$ and $M^n=S^n_{r}(0)$.
\end{Corollary}

Proof of Theorem \ref{t7}.

\begin{proof}
Since $M^n$ is a compact hypersurface, the potential function achieves its maximum in some $p\in M$. From the fact that the gradient of $f$ points in the direction of greatest increase  and $\nabla f(p)=0$, we see that $\bar{\nabla}f(p)$ and $N(p)$ have the same direction. Therefore
\begin{eqnarray*}
	H_f+|\bar{\nabla}f|(p) &=& H_f+\langle\bar{\nabla}f(p),N(p)\rangle\\
						  &=&H(p).
\end{eqnarray*}
Using the fact that $\left|\bar{\nabla}f\right|^{2}+\bar{R}-2\lambda f=C$ and the hypothesis on the mean curvature, we get
\begin{eqnarray*}
	\sqrt{2\lambda f(p)+C-\bar{R}}=|\bar{\nabla}f|(p)&\leq& \frac{2H_f-|H_f|+\sqrt{H_f^2+4a}}{2}-H_f\\
								                     &=& \frac{-|H_f|+\sqrt{H_f^2+4a}}{2}.
\end{eqnarray*}
Therefore
\begin{eqnarray*}
	f\leq f(p)\leq \frac{1}{2\lambda}\left\lbrace \left(\frac{-|H_f|+\sqrt{H_f^2+4a}}{2}\right)^2+\bar{R}-C\right\rbrace,
\end{eqnarray*}
which implies that $M^n$ lies in $\overline{D^-(|H_f|,a)}$. By Theorem \ref{t3}, we conclude that $M^n\subseteq\partial D^-(|H_f|,a)$.
\end{proof}

Proof of Corollary \ref{t10}.
\begin{proof}
Take $a=\frac{n-k}{2}$. As in the proof of Corollary \ref{c1} Item (a) we have $tr_{M^{n}}\bar{\nabla}\bar{\nabla}f\geq a$ and $\partial D^-(H_f,a)=S^{n-k}_r\left(0\right)\times S_{\sqrt{2\left(k-1\right)}}^{k}(0)$ where $r=-H_{f}+\sqrt{H_{f}^{2}+2\left(n-k\right)}$. The result follows from Theorem \ref{t7}.
\end{proof}

Proof of Theorem \ref{t9}.

\begin{proof}
	We compute
	\begin{eqnarray*}
		\Delta f&=& \text{tr}_M\bar{\nabla}\bar{\nabla} f -\langle \bar{\nabla}f,N\rangle H\\ 
		&=& \text{tr}_M\bar{\nabla}\bar{\nabla}f+(H_f-H)H\\
		&\leq& b +(H_f-H)H\\
		&=& \frac{H_f^2+4b}{4}-\left(H-\frac{H_f}{2}\right)^2\leq 0.
	\end{eqnarray*}
	In the third line we used the assumption on $\bar{\nabla} \bar{\nabla} f$ and in the fourth line we used the assumption on $H$. Let $\varphi \in C_c^\infty(M^n)$. Integrating by parts, we have
	\begin{eqnarray*}
		\int_M \varphi^2|\nabla f|^2 e^{- f}&\leq& \int_M \varphi^2(|\nabla f|^2-\Delta f)e^{- f}\\
		&=& -\int_M \varphi^2\Delta_{ f} fe^{- f}\\
		&=& \int_M \langle \nabla \varphi^2,\nabla f\rangle e^{- f}\\
		&=& 2\int_M\langle \nabla \varphi, \varphi\nabla f\rangle e^{- f}\\
		&\leq& \int_M \left[ 2|\nabla \varphi|^2+\frac{1}{2}\varphi^2|\nabla f|^2\right]e^{- f}.
	\end{eqnarray*}
	Therefore
	\[
	\frac{1}{2}\int_M \varphi^2|\nabla f|^2 e^{- f}\leq 2\int_M|\nabla \varphi|^2e^{- f}.
	\]
	Fix $p_0\in M^n$ and consider a sequence $\varphi_j \in C_0^\infty(M^n)$ such that $\varphi_j=1$ on $B^M_j(p_0)$, $\varphi_j=0$ on $M^n\setminus B^M_{2j}(p_0)$ and $|\nabla \varphi_j|\leq \frac{1}{j}$. Using the monotone convergence theorem and the fact that the weighted volume is finite we have
	\[
	\int_M |\nabla f|^2 e^{- f}=0,
	\]
	which implies that $f$ is constant on $M^n$.
\end{proof}

Using Theorem \ref{t9}, Lemma \ref{remark1} and Theorem \ref{lema1} we obtain a new result for ambient spaces which are shrinking Ricci solitons with constant scalar curvature. Note that taking $b=\frac{n}{2}$, $\lambda = \frac{1}{2}$, $\bar{R} = 0$ and $C=0$ in the next result we recover Theorem 1.2 in \cite{cheng2016rigidity}.
\begin{Corollary}\label{corolult}
	Let $\bar{M}^{n+1}$ be a shrinking Ricci soliton with constant scalar curvature and let $M^n$ be a complete CWMC hypersurface properly immersed in $\bar{M}^{n+1}$. Suppose that $\text{tr}_{M}\bar{\nabla}\bar{\nabla}f\leq b$ for some $b> 0$. If
	\[
	H\geq \frac{H_f+\sqrt{H_f^2+4b}}{2},
	\]
	then $M^n\subseteq f^{-1}(\gamma)$ for some $\gamma>0$. Moreover, if $f^{-1}(\gamma)$ is connected and complete then $M^n=f^{-1}(\gamma)$ where
	\[
	\gamma= \frac{1}{2\lambda}\left\lbrace \frac{1}{4}\left(-H_f+\sqrt{H_f^2+4(n\lambda-\bar{R})}\right)^2+\bar{R}-C\right\rbrace.
	\]
\end{Corollary}

\begin{proof}
Since $M^{n}$ is properly immersed we know that $M^{n}$ has finite weighted volume (see Lemma \ref{remark1}). Using Theorem \ref{t9} and Lemma \ref{remark1} we have $M^n \subseteq f^{-1}(\gamma)$. Suppose that $f^{-1}(\gamma)$ is connected and complete. Then $M^n=f^{-1}(\gamma)$. Using Theorem \ref{lema1} we get the conclusion of $\gamma$.
\end{proof}

Proof of Corollary \ref{corin3}.

\begin{proof}
Assuming that the conclusion of the result is false we will show that
this leads to a contradiction. Take $b=\frac{n+1-k}{2}$. If the conclusion is false we have
\[
H\geq\frac{H_{f}+\sqrt{H_{f}^{2}+4b}}{2}.
\]
As in the proof Corollary \ref{c1} we have $\text{tr}_{M}\bar{\nabla}\bar{\nabla}f\leq b$, $\lambda=\frac{1}{2}$ and $C=\bar{R}=\frac{k}{2}$.
Using Corollary \ref{corolult} we have $M^n = f^{-1}(\gamma)$ where
\[
\gamma=\frac{1}{4}\left(-H_{f}+\sqrt{H_{f}^{2}+2\left(n-k\right)}\right)^2.
\]
Using the fact that $f(x,y)=\frac{\left|x\right|^{2}}{4}$ we see that $M^{n}=S_{r}^{n-k}\left(0\right)\times S_{\sqrt{2\left(k-1\right)}}^{k}\left(0\right)$
with $r=-H_{f}+\sqrt{H_{f}^{2}+2\left(n-k\right)}$. However for $M^n$ we have (see Corollary \ref{t10})
$$H = \frac{H_{f}+\sqrt{H_{f}^{2}+2(n-k)}}{2}<\frac{H_{f}+\sqrt{H_{f}^{2}+2(n+1-k)}}{2}=\frac{H_{f}+\sqrt{H_{f}^{2}+4b}}{2},$$
a contradiction.
\end{proof}

\section{Constant weighted mean curvature hypersurfaces in $\mathbb{R}^{n+1}$}\label{s4}

In this section we prove Theorem \ref{t4} and Corollary \ref{t5}.

Proof of Theorem \ref{t4}.
\begin{proof}
By \cite{guang2018gap} we have
	\[
	\Delta_f|A|^2=2\left(\frac{1}{2}-|A|^2\right)|A|^2-2H_f\text{tr}A^3+2|\nabla A|^2.
	\]
	Considering $q\geq 1$ as in the hypothesis, we have
	\begin{eqnarray*}
		|A|^q\Delta_f|A|^2&=&2\left(\frac{1}{2}-|A|^2\right)|A|^{q+2}-2|A|^qH_f\text{tr}A^3+2|A|^q|\nabla A|^2\\
		&\geq&2\left(\frac{1}{2}-|A|^2\right)|A|^{q+2}-2|H_f||A|^{q+3}+2|A|^q|\nabla A|^2\\
		&=&|A|^{q+2}(1-2|A|^2-2|H_f||A|)+2|A|^q|\nabla A|^2.
	\end{eqnarray*}
	Note that since 
	\[
	|A|\leq\frac{\sqrt{H_f^2+2}-|H_f|}{2},
	\]
	we have that $1-2|A|^2-2|H_f||A|\geq0$. Therefore
	\[
	|A|^q\Delta_f|A|^2\geq 2|A|^q|\nabla A|^2.
	\]
	Let $\varphi\in C_c^\infty(M^n)$. Integrating by parts we have
	\begin{eqnarray*}
		2\int_M \varphi^2|A|^q|\nabla A|^2e^{-f}&\leq& \int_M \varphi^2|A|^q\Delta_f|A|^2e^{-f}\\
		&=& -\int_M \langle \nabla \varphi^2|A|^q, \nabla |A|^2\rangle e^{-f}\\
		&=& -2\int_M q|A|^{q}\varphi^2|\nabla |A||^2e^{-f}-4\int_M \langle |A|^{\frac{q}{2}+1}\nabla \varphi, \varphi|A|^{\frac{q}{2}}\nabla|A|\rangle e^{-f}\\
		&\leq& -2\int_M q|A|^{q}\varphi^2|\nabla |A||^2e^{-f}+4\int_M  |A|^{\frac{q}{2}+1}|\nabla \varphi| \varphi|A|^{\frac{q}{2}}|\nabla|A|| e^{-f}.
	\end{eqnarray*}
	Here in the last line we use the Cauchy-Schwarz inequality. From the identity $2ab\leq \varepsilon a^2+\frac{1}{\varepsilon}b^2$, choosing $a=|A|^{\frac{q}{2}+1}|\nabla \varphi|$ and $b=\varphi|A|^{\frac{q}{2}}|\nabla |A||$ we have
	\[
	2\int_M \varphi^2|A|^q|\nabla A|^2e^{-f}\leq 2\varepsilon\int_M |A|^{q+2}|\nabla \varphi|^2e^{-f}+\left(\frac{2}{\varepsilon}-2q\right)\int_M \varphi^2|A|^q|\nabla|A||^2e^{-f}.
	\]
	Choosing $\varepsilon=\frac{1}{q}$, we get
	\begin{eqnarray*}
		\int_M \varphi^2|A|^q|\nabla A|^2e^{-f}\leq \frac{1}{q}\int_M |A|^{q+2}|\nabla \varphi|^2e^{-f}.
	\end{eqnarray*}
	Since $|A|^2\leq C$, we have 
	\[
	\int_M \varphi^2|A|^q|\nabla A|^2e^{-f}\leq \frac{C}{q}\int_M |A|^{q}|\nabla \varphi|^2e^{-f}.
	\]
	Let us fix $p_0\in M$ and consider the sequence $\varphi_j$ such that $\varphi_j=1$ on $B^M_j(p_0)$, $\varphi_j=0$ on $M\setminus B^M_{2j}(p_0)$ and $|\nabla \varphi_j|\leq \frac{1}{j}$. Using the monotone convergence theorem and the assumption that $A\in L^q_f(M)$, we conclude that
	\[
	|A||\nabla A|=0.
	\]
	Let $\mathcal{A}=\lbrace x\in M; |A|(x)=0\rbrace$. The set $M\setminus \mathcal{A}$ is open and since
	\[
	|A||\nabla |A||\leq |A||\nabla A|=0,
	\]
	we see that $|A|$ is constant on $M\setminus \mathcal{A}$. If $\mathcal{A}\neq \emptyset$, then using continuity we conclude that $|A|=0$, which implies that $M^n$ is a hyperplane. If $\mathcal{A}= \emptyset$, then $|\nabla A|=0$. In this case using Theorem 4 in Lawson \cite{lawson1969local} and the fact that $M^n$ is complete we conclude that $M^n$ is a generalized cylinder.	
\end{proof}

Proof of Corollary \ref{t5}.

\begin{proof}
	We estimate the Bakry-Emery-Ricci tensor $Ric_f$ of a CWMC hypersurface. Let $p\in M$ and choose a orthonormal basis of $T_pM$ such that $\nabla_{e_j}e_i(p)=0$. Let $h_{ij}=A(e_i,e_j)$. We have
	\begin{eqnarray*}
		(Ric_f)_{ij}&=&R_{ij}+\nabla_{e_i}\nabla_{e_j}f\\
	                &=& R_{ij}+\bar{\nabla}_{e_i}\bar{\nabla}_{e_j}f+\langle \bar{\nabla}f,A(e_i,e_j)N\rangle\\
	                &=& R_{ij}+\frac{1}{2}\delta_{ij}+\frac{\langle x,N\rangle}{2}h_{ij}\\
		            &=&-Hh_{ij}-\sum_{l=1}^nh_{il}h_{lj}+\frac{\langle x,N\rangle}{2}h_{ij}+\frac{1}{2}\delta_{ij}\\
	             	&=&-H_fh_{ij}-\sum_{l=1}^nh_{il}h_{lj}+\frac{1}{2}\delta_{ij}\\
	            	&\geq&-|H_f||h_{ij}|-\sum_{l=1}^nh_{il}h_{lj}+\frac{1}{2}\delta_{ij}\\
	            	&\geq&-|H_f||A|-\sum_{l=1}^nh_{il}h_{lj}+\frac{1}{2}\delta_{ij}.
	\end{eqnarray*}
	In the fourth line we used the Gauss equation. Since $\sup |A|<\frac{\sqrt{H_f^2+2}-|H_f|}{2}$ we have
	\[
	Ric_f\geq -|A|^2-|H_f||A|+\frac{1}{2}\geq -(\sup|A|)^2-|H_f|\sup|A|+\frac{1}{2}>0.
	\]
	By Theorem 4.1 in \cite{wei2009comparison}, the condition above implies that $M^n$ has finite weighted volume. Since $|A|^q$ is bounded and $M^n$ has finite weighted volume, we see that $|A|\in L_f^q(M)$. Applying Theorem \ref{t4} we conclude that $M^n$ is a hyperplane.
\end{proof}

\bibliographystyle{abbrv}
\bibliography{Biblio1}

\end{document}